\documentclass[11pt]{amsart}

\usepackage[T1]{fontenc}
\usepackage[utf8]{inputenc}
\usepackage{lmodern}
\usepackage{amsmath,amssymb,amsthm,mathtools}
\usepackage{enumitem}
\usepackage{tikz-cd}
\usepackage{hyperref}

\newtheorem{theorem}{Theorem}[section]
\newtheorem{proposition}[theorem]{Proposition}
\newtheorem{lemma}[theorem]{Lemma}
\newtheorem{corollary}[theorem]{Corollary}
\newtheorem{question}[theorem]{Question}

\theoremstyle{definition}
\newtheorem{definition}[theorem]{Definition}
\newtheorem{example}[theorem]{Example}
\newtheorem{notation}[theorem]{Notation}
\newtheorem{remark}[theorem]{Remark}

\newcommand{\coker}{\operatorname{coker}}
\newcommand{\A}{\mathcal A}
\newcommand{\W}{\mathcal W}

\newcommand{\B}{\mathcal B}

\newcommand{\T}{\mathsf T}
\newcommand{\F}{\mathsf F}
\newcommand{\tors}{\operatorname{tors}}
\newcommand{\torf}{\operatorname{torf}}
\newcommand{\filt}{\operatorname{filt}}
\newcommand{\gen}{\operatorname{gen}}
\newcommand{\cogen}{\operatorname{cogen}}

\newcommand{\Hom}{\operatorname{Hom}}
\newcommand{\Ext}{\operatorname{Ext}}
\newcommand{\Up}{\operatorname{Up}}
\newcommand{\Low}{\operatorname{Low}}

\newcommand{\elll}{\ell}

\newcommand{\lperpzero}{{{}^{\perp_0}}}

\newcommand{\tGen}[1]{\mathbf{T}(#1)}
\newcommand{\tpair}[2]{(\mathbf{#1}, \mathbf{#2})  }

\hypersetup{
    colorlinks=true,
    linkcolor=blue,
    pdftitle={Wide subcategories and brick-finiteness for length categories},
    pdfauthor={Francesco Sentieri},
    pdfpagemode=UseOutlines,
    }

\title[Wide subcategories and brick-finiteness]{Wide subcategories and brick-finiteness for length categories}
\author{Francesco Sentieri}
\email{francesco.sentieri@guarinoveronese.it}
\keywords{Torsion pairs, bricks, wide subcategories, abelian length category}

\begin{document}

\begin{abstract}
We extend to abelian length categories of finite rank a characterisation of brick-finiteness known for finite-dimensional algebras. We prove that such a category is brick-finite if and only if every torsion class is generated by a wide subcategory and every torsionfree class is cogenerated by a wide subcategory. The proof is based on an exchange relation for brick labels across wide intervals which is a shadow of the mutation of simple minded collections. As a corollary, we extend the validity of the first Brick Brauer-Thrall conjecture to this setting.
\end{abstract}

\maketitle

\tableofcontents

\section{Introduction}

In recent years, following \cite{AIR}, torsion pairs returned to a central position in the representation theory of finite-dimensional algebras. While in the seminal work of Adachi, Iyama and Reiten torsion pairs were studied with techniques exclusively applicable in the context of artin algebras, in particular through the concept of functorial finiteness and mutation, the connections of torsion pairs with bricks, obtained in \cite{DIJ}, paved the way to more general lattice-theoretical techniques.

In \cite{DIRRT}, the authors have shown, with arguments valid in the generality of abelian length categories, that the lattice of torsion classes enjoys several nice properties, namely it is bi-algebraic and completely semidistributive, and its local structure can be controlled by means of brick labels. These brick labels can be characterised precisely, as done in \cite{BCZ}, as the so-called minimal extending and co-extending objects. The complete description of torsion pairs by means of bricks was then finally obtained in \cite{EnomotoMonobrick}, where the author built on the notion of simple object in an exact category to obtain a bijection between torsionfree classes and special collections of objects, called monobricks.

The study of bricks also led to the introduction of a new class of algebras, in analogy with the classical family of representation-finite algebras: the \(\tau\)-tilting finite algebras. These are artin algebras with a finite number of torsion pairs, equivalently of isomorphism classes of bricks, so they are also called brick-finite. Several interesting characterisations of these algebras were proven in the literature, starting with \cite{DIJ}, where the concept was introduced, up to some results in analogy with the first Brauer--Thrall conjecture, see \cite{SchrollTreffinger} and \cite{MousavandPaquetteBT}.

More recently, in the second version of the survey \cite{RingelBrickChain}, appeared a characterisation using only elementary properties of the lattice and valid for general abelian length categories: brick-finiteness is equivalent to requiring all the elements of the lattices of torsion and torsionfree classes to be compact. Torsion pairs where both the components are compact in their respective lattice were studied in the recent preprint \cite{AsaiBicompact}, where such pairs are conjectured to correspond to functorially finite pairs in the artin algebra setting.

The aim of this note is to give an elementary proof of a further characterisation of brick-finiteness in the length category context. Recall that a full subcategory is called wide if it is closed under extensions, kernels and cokernels. The connection between torsion pairs and wide subcategories was first explored for hereditary algebras in \cite{IngallsThomas}, then extended to generic finite-dimensional algebras in \cite{MarksStovicek}, where it was shown that for a brick-finite algebra there is a bijection between wide subcategories and torsion classes. A classical result of Ringel gives a bijection between wide subcategories and Hom-orthogonal collections of bricks; such collections, now called semibricks, are the object of \cite{AsaiSemibricks}.

Our main result is the following.
\begin{theorem}\label{thm:intro-main}
Let \(\A\) be an abelian length category with a finite number of simple objects up to isomorphism. Then \(\A\) is brick-finite if and only if for every torsion pair \((\mathbf t,\mathbf f)\) in \(\A\) there exist wide subcategories \(\W\) and \(\W'\) with
\[
\mathbf t=\T(\W),\qquad \mathbf f=\F(\W').
\]
\end{theorem}

Notice that the new result is stronger than the one using compactness in the following way: assuming every torsion or torsionfree class is compact amounts to assuming that they are all generated or cogenerated by wide subcategories and that every semibrick involved is finite. Our characterisation requires only the existence of the relevant wide subcategories and finiteness of the semibrick of simple objects. As a further remark, the fact that a torsion class is widely generated admits a nice interpretation at the level of the heart of the associated HRS-tilt, for this we refer to \cite[Proposition 5.4]{Sh}.

In the setting of artin algebras it is known that a weaker condition implies brick-finiteness: namely assuming that every torsionfree class is cogenerated by a wide subcategory (\cite[Proposition 5.4]{WideCoreflective}). This is related to the existence of locally minimal/locally maximal non-functorially finite torsion classes for brick-infinite algebras, see \cite{SentieriAuslander}. 

The proof is organised around an exchange relation for brick labels, which seems to be a shadow of the mutation of simple minded collections or more generally semibrick pairs in the derived category (originally introduced in \cite{SimpleMindOrig} for finite-dimensional algebras): 

 As already noticed by \cite{AsaiPfeifer}, some inclusions of torsion classes, the so-called wide intervals, resemble the mutation we have between functorially finite torsion pairs. This was made more precise, when the notion of cosilting mutation was introduced in \cite{ALSV}. 
Assume that
\[
\mathbf t \subsetneq \mathbf u
\]
is such that \( \mathbf{u} \cap \mathbf{t}^{\perp_0} \) is a wide subcategory. Then, under the wide generation/cogeneration hypotheses, it is possible to show that  we have a bijective correspondence between neighbours of \( \mathbf{t} \)  and neighbours of \( \mathbf{u} \) in the Hasse quiver of \( \tors(\A) \). This bijection can be explicitly computed in terms of suitable approximation sequences for brick labels.

This bijection also recovers the regularity of the Hasse diagram of \( \tors(\A) \) for brick-finite categories.

An immediate consequence of the main theorem is the following generalisation of the first-Brick Brauer Thrall Theorem 

\begin{theorem}\label{thm:intro-bbt}
Let \(\A\) be an abelian length category with a finite number of simple objects up to isomorphism. Then \(\A\) is brick-finite if and only if there exists some \(N\in\mathbb N\) such that the length of every brick in \(\A\) is less than \(N\).
\end{theorem}

\noindent\textbf{Acknowledgments.} The author wishes to thank Lidia Angeleri H\"ugel for her valuable comments and suggestions.

\begin{notation}\label{not:standing}
Throughout this paper, \(\A\) denotes an abelian length category, i.e. an essentially small abelian category where every object admits a (finite) composition series. We say that \(\A\) has finite rank if it has only finitely many simple objects up to isomorphism. Unless otherwise stated, all subcategories are full and closed under isomorphism.
An object \( B \) is called a brick if its endomorphism ring is a skew-field.

For a class of objects \(\mathcal S\) we denote by:
\begin{itemize}[leftmargin=2.5em]
\item \(\filt(\mathcal S)\) the full subcategory of all objects admitting a finite filtration with factors in \(\mathcal S\);
\item \(\gen(\mathcal S)\) the full subcategory of all quotients of finite direct sums of objects in \(\mathcal S\);
\item \(\cogen(\mathcal S)\) the full subcategory of all subobjects of finite direct sums of objects in \(\mathcal S\);
\item \(\lperpzero\mathcal S=\{A\in\A\mid \Hom_\A(A,S)=0\text{ for all }S\in\mathcal S\}\);
\item \(\mathcal S^{\perp_0}=\{A\in\A\mid \Hom_\A(S,A)=0\text{ for all }S\in\mathcal S\}\);
\item \(\T(\mathcal S)=\filt\gen(\mathcal S)\), the smallest torsion class containing \(\mathcal S\);
\item \(\F(\mathcal S)=\filt\cogen(\mathcal S)\), the smallest torsionfree class containing \(\mathcal S\).
\end{itemize}
We write \(\elll(X)\) for the composition length of an object \(X\in\A\).
\end{notation}

\section{Preliminaries}

We recall some preliminary results from the literature. For a self-contained survey on torsion pairs and bricks in the generality of abelian length categories we refer to \cite{RingelBrickChain}. For a more global account on recent developments and connections with adjacent topics, like stability conditions, we suggest \cite{MousavandPaquetteSurvey}. The notion of torsion pair was originally introduced in \cite{Dickson}.

\subsection{Torsion pairs and relative simples}

\begin{definition}
A pair of subcategories \((\mathbf t,\mathbf f)\) of \(\A\) is a torsion pair if:
\begin{enumerate}[label=\textup{(\arabic*)},leftmargin=2.5em]
\item \(\Hom_\A(T,F)=0\), for all \(T\in\mathbf t\) and all \(F\in\mathbf f\);
\item for every object \(M\in\A\), there exists a short exact sequence
\[
0\longrightarrow T\longrightarrow M\longrightarrow F\longrightarrow 0
\]
with \(T\in\mathbf t\) and \(F\in\mathbf f\).
\end{enumerate}
In a torsion pair, \(\mathbf t\) is called the torsion class and \(\mathbf f\) is called the torsionfree class.
\end{definition}

Recall that a subcategory is a torsion class if and only if it is closed under quotients and extensions, and it is a torsionfree class if and only if it is closed under subobjects and extensions. Moreover
\[
\mathbf t={}^{\perp_0}\mathbf f,
\qquad
\mathbf f=\mathbf t^{\perp_0}.
\]

\begin{definition}
Let \((\mathbf t,\mathbf f)\) be a torsion pair. A non-zero object \(F\in\mathbf f\) is said to be \(\mathbf f\)-simple if all its proper quotients are in \(\mathbf t\). Dually, a non-zero object \(T\in\mathbf t\) is said to be \(\mathbf t\)-simple if all its proper subobjects are in \(\mathbf f\).
\end{definition}

\begin{lemma}[{\cite{EnomotoMonobrick}}]\label{lem:relative-simple-brick}
The following statements hold.
\begin{enumerate}[label=\textup{(\arabic*)},leftmargin=2.5em]
\item If an object is \(\mathbf f\)-simple for some torsion pair, then it is a brick.
\item If \(B\) is a brick, then \(B\) is \(\F(B)\)-simple.
\item An object \(B\neq 0\) is \(\mathbf f\)-simple if and only if every non-zero map \(B\to F\) with \(F\in\mathbf f\) is a monomorphism.
\end{enumerate}
Dually, a non-zero object \(B\in\mathbf t\) is \(\mathbf t\)-simple if and only if every non-zero map \(T\to B\) with \(T\in\mathbf t\) is an epimorphism.
\end{lemma}

\begin{theorem}[{\cite[Theorem 2.11]{EnomotoMonobrick}}]\label{thm:f-simples-filter}
Every torsionfree class \(\mathbf f\) is determined by its \(\mathbf f\)-simple objects. More precisely, every object in \(\mathbf f\) admits a filtration by \(\mathbf f\)-simples.
\end{theorem}

\begin{definition}
 A set of bricks $ \mathcal{B} $ is a \emph{semibrick} if non-isomorphic elements of $ \mathcal{B} $ are Hom-orthogonal.
\end{definition}

\begin{definition}[\cite{AngeleriHerzogLaking}, \cite{BCZ}]
Let \((\mathbf t,\mathbf f)\) be a torsion pair. An object \(B\) is torsionfree, almost torsion with respect to \((\mathbf t,\mathbf f)\) if:
\begin{enumerate}[label=\textup{(\arabic*)},leftmargin=2.5em]
\item \(B\) is \(\mathbf f\)-simple;
\item for every object \(F\in\mathbf f\) and every map \(B\to F\), the cokernel is in \(\mathbf f\).
\end{enumerate}
Equivalently, condition \textup{(2)} can be replaced by:
\begin{enumerate}[label=\textup{(2')},leftmargin=2.5em]
\item for every non-split short exact sequence
\[
0\longrightarrow B\longrightarrow M\longrightarrow T\longrightarrow 0
\]
with \(T\in\mathbf t\), one has \(M\in\mathbf t\).
\end{enumerate}
Dually, one defines torsion, almost torsionfree objects.
\end{definition}

Torsionfree, almost torsion objects are the maximal relative simples in the following sense.

\begin{proposition}[{\cite[Lemma 4.4]{EnomotoMonobrick}}]\label{prop:maximal-f-simple}
An \(\mathbf f\)-simple object \(B\) is torsionfree, almost torsion if and only if, for every \(\mathbf f\)-simple object \(B'\), we have
\[
\Hom_\A(B,B')\neq 0
\quad\Longleftrightarrow\quad
B'\simeq B.
\]
In particular, the torsionfree, almost torsion objects for a torsion pair form a semibrick. Dually, the torsion, almost torsionfree objects form a semibrick.
\end{proposition}

We will use the following Ext-orthogonality several times. 

\begin{lemma}\label{lem:ext-orthogonality}
Let \((\mathbf t,\mathbf f)\) be a torsion pair, let \(T\) be torsion, almost torsionfree, and let \(F\) be torsionfree, almost torsion for this torsion pair. Then
\[
\Ext^1_\A(T,F)=0.
\]
\end{lemma}

\begin{proof}
We give a short proof for the convenience of the reader. Let 
\[
0 \to F \to M \to T \to 0
\]
be a short exact sequence. Then, if the sequence were not split, as $ F $ is torsionfree, almost torsion, $ M $ would have to be in $ \mathbf{t} $. But then, as $ T $ is torsion, almost torsionfree, we would have that $ F \in \mathbf{t} $. This would imply $ F = 0 $, which is a contradiction.
\end{proof}

\subsection{Wide subcategories and compactness}

\begin{definition}
A subcategory \(\W\subseteq\A\) is wide if it is closed under kernels, cokernels and extensions.

A torsion pair \((\mathbf t,\mathbf f)\) is widely cogenerated if \(\mathbf f=\F(\W)\) for some wide subcategory \(\W\). Dually, it is widely generated if \(\mathbf t=\T(\W)\) for some wide subcategory \(\W\).
\end{definition}

\begin{theorem}[{\cite[Section 1.2]{RingelSpecies}}]\label{thm:ringel-semibricks}
There is a bijection between semibricks in \(\A\) and wide subcategories of \(\A\). This bijection associates with a semibrick \(\mathcal S\) the extension closure \(\filt(\mathcal S)\). Its inverse associates with every wide subcategory the collection of its simple objects.
\end{theorem}

\begin{proposition}[{\cite[Proposition 5.2]{EnomotoMonobrick}}]\label{prop:wide-cogenerated-criterion}
Let \((\mathbf t,\mathbf f)\) be a torsion pair. The following statements are equivalent.
\begin{enumerate}[label=\textup{(\arabic*)},leftmargin=2.5em]
\item \((\mathbf t,\mathbf f)\) is widely cogenerated.
\item Every \(\mathbf f\)-simple object is a subobject of some torsionfree, almost torsion object for \((\mathbf t,\mathbf f)\).
\item \((\mathbf t,\mathbf f)\) is widely cogenerated by the extension closure of its torsionfree, almost torsion objects.
\end{enumerate}
Dually, \((\mathbf t,\mathbf f)\) is widely generated if and only if every \(\mathbf t\)-simple object is a quotient of some torsion, almost torsionfree object.
\end{proposition}

\begin{proposition}[{\cite[Corollary 5.6]{RingelBrickChain}}]\label{prop:finitely-cogenerated}
Let \((\mathbf t,\mathbf f)\) be a torsion pair. Then \(\mathbf f\) is finitely cogenerated, that is \(\mathbf f=\F(M)\) for some object \(M\), if and only if \((\mathbf t,\mathbf f)\) is widely cogenerated and has only finitely many torsionfree, almost torsion objects.

Dually, \(\mathbf t\) is finitely generated if and only if \((\mathbf t,\mathbf f)\) is widely generated and has only finitely many torsion, almost torsionfree objects.
\end{proposition}

\begin{definition}
We say that \(\A\) is widely determined, or WD, if every torsion class in \(\A\) is widely generated. We say that \(\A\) is widely co-determined, or WCD, if every torsionfree class in \(\A\) is widely cogenerated.
\end{definition}

\begin{corollary}\label{cor:basic-wide-consequences}
Assume \(\A\) is WD and WCD.
\begin{enumerate}[label=\textup{(\arabic*)},leftmargin=2.5em]
\item Every non-zero torsion class has a torsion, almost torsionfree object.
\item If a torsion class is not compact, then it has infinitely non-isomorphic torsion, almost torsionfree objects.
\item If a torsion class is not co-compact, then it has infinitely many non-isomorphic torsionfree, almost torsion objects.
\end{enumerate}
\end{corollary}

\begin{proof}
For (1), let \(\mathbf t\neq 0\) and write \(\mathbf t=\T(\W)\) with \(\W\) wide. The wide subcategory \(\W\) is non-zero, hence has a simple object. By Theorem \ref{thm:ringel-semibricks} and the dual of Proposition \ref{prop:wide-cogenerated-criterion}, this gives a torsion, almost torsionfree object.

For (2), if there were only finitely many torsion, almost torsionfree objects, then Proposition \ref{prop:finitely-cogenerated} would imply that \(\mathbf t\) is finitely generated, hence compact. This is a contradiction. The statement (3) is obtained by duality.
\end{proof}

\subsection{The lattices of torsion and torsionfree classes}

Let \(\torf(\A)\) be the set of torsionfree classes ordered by inclusion and let \(\tors(\A)\) be the set of torsion classes ordered by inclusion.

\begin{theorem}[{\cite[Theorem 3.1]{DIRRT}}]\label{thm:torsionfree-lattice}
The poset \((\torf(\A),\subseteq)\) is a complete lattice. Its meet is given by intersections and its join is given by taking the smallest torsionfree class containing the union. Moreover, this lattice is bi-algebraic and completely semidistributive. Dually, \(\tors(\A)\) is a complete lattice.
\end{theorem}

The following elementary observation allows an immediate dualization of all the lattice theoretical results about torsionfree classes.

\begin{lemma}
The assignment:
\[
	\torf(\A) \to \tors(\A), \quad \mathbf{f} \mapsto {}^{\perp_0}\mathbf{f}
\]
is a duality (i.e. an anti-isomorphism) between the lattice of torsionfree classes and the lattice of torsion classes. 
Moreover the assignment:
\[
	\torf(\A) \to \tors(\A^{\mathrm{op}}), \quad \mathbf{f} \mapsto \mathbf{f}
\]
is a lattice isomorphism.
\end{lemma}

\begin{definition}
An element \( x \) of a complete lattice is said to be \emph{compact} if whenever \( x \le \bigvee_I y_i \) for some collection \( \{ y_i \}_{i \in I} \) of elements of \( L \) there exists a finite subset \( I_0 \subset I \) such that \( x \le \bigvee_{I_0} y_i \). 

Dually, an element \( x \) is \emph{cocompact} if whenever \( x \ge \bigwedge_I y_i \)  for some collection \( \{ y_i \}_{i \in I} \) of elements of \( L \) there exists a finite subset \( I_0 \subset I \) such that \( x \ge \bigwedge_{I_0} y_i \). 
\end{definition}

\begin{remark}
Notice that in a bialgebraic lattice the inequalities in the definition of compact and cocompact elements could be replaced by equalities.

Moreover, by the duality between torsion and torsionfree classes, a torsion class is cocompact if and only if the corresponding torsionfree class is compact. 
\end{remark}

\begin{proposition}[{\cite[Proposition 3.2]{DIRRT}}]\label{prop:compact-finitely-generated}
A torsionfree class is compact in \(\torf(\A)\) if and only if it is finitely cogenerated. Dually, a torsion class is compact in \(\tors(\A)\) if and only if it is finitely generated.
\end{proposition}

\begin{proposition}[{\cite[Theorem 1.0.2, Theorem 2.3.2]{BCZ}}]\label{prop:brick-labels-covers}
Let \((\mathbf t,\mathbf f)\) be a torsion pair.
\begin{enumerate}[label=\textup{(\arabic*)},leftmargin=2.5em]
\item Upper covers of \(\mathbf t\) in \(\tors(\A)\) are in bijection with (isomorphism classes of) torsionfree, almost torsion objects for \((\mathbf t,\mathbf f)\).
\item Lower covers of \(\mathbf t\) in \(\tors(\A)\) are in bijection with (isomorphism classes of) torsion, almost torsionfree objects for \((\mathbf t,\mathbf f)\).
\end{enumerate}
If \(\mathbf t\lessdot \mathbf u\) is a cover relation, the corresponding object is called the brick label of the cover. It is torsionfree, almost torsion for \(\mathbf t\) and torsion, almost torsionfree for \(\mathbf u\).
\end{proposition}

We will use the following notation for labels.

\begin{definition}
For a torsion pair \((\mathbf t,\mathbf f)\), let
\[
\Up(\mathbf t)=\{\text{torsionfree, almost torsion objects for }(\mathbf t,\mathbf f)\}/\cong,
\]
and let
\[
\Low(\mathbf t)=\{\text{torsion, almost torsionfree objects for }(\mathbf t,\mathbf f)\}/\cong.
\]
Thus \(\Up(\mathbf t)\) is the set of upper labels of \(\mathbf t\), while \(\Low(\mathbf t)\) is the set of lower labels of \(\mathbf t\).

We will sometimes abuse notation, for instance writing $ R \in \Up(\mathbf{t}) $ for an object \( R \) meaning that \( R \) is torsionfree, almost torsion for \( \tpair{t}{f} \).
\end{definition}

We now recall a fundamental result and notion from the work of Asai which leads to some sort of mutation process in the context of abelian length category:

\begin{definition}
Let \((\mathbf t,\mathbf f)\) and \((\mathbf u,\mathbf v)\) be torsion pairs with \( \mathbf{t} \subsetneq \mathbf{u} \), we say that they form a \emph{wide interval} if \( \mathbf{f} \cap \mathbf{u} \) is a wide subcategory.
\end{definition}

\begin{remark}
Notice that if \( \mathbf{u} \) covers \( \mathbf{t} \), then the corresponding interval is wide: in fact it has the form \( \filt(B) \) for the brick \( B \) labeling the cover ( see for instance \cite[Theorem 3.3]{DIRRT} ).
\end{remark}

We will use the following notation:

\begin{definition}
For two full subcategories \( \mathcal{L} \) and \( \mathcal{R} \) of \( \A \) we fix
\[
\mathcal{L} * \mathcal{R} := \left\{ M \in \A \,|\,  0 \to L \to M \to R \to 0, \text{ with }L \in \mathcal{L}\text{ and }R \in \mathcal{R} \right\}
\]
\end{definition}

\begin{proposition}[{\cite[Theorem 4.2, Proposition 6.3]{AsaiPfeifer}}]\label{prop:cover-intersection}
Let \((\mathbf t,\mathbf f)\) and \((\mathbf u,\mathbf v)\) be torsion pairs forming a wide interval and let \( \mathcal{W} = \mathbf{u} \cap \mathbf{f} \). Then
\[
\mathbf{t} = \mathbf{u} \cap {}^{\perp_0} \mathcal{W},
\qquad
\mathbf f=\mathcal{W}*\mathbf v.
\]
Dually, 
\[
\mathbf {u} = \mathbf{t} * \mathcal{W}, \qquad \mathbf{f} \cap \mathcal{W}^{\perp_0} = \mathbf{v}
\]
Moreover, the simple objects of \( \mathcal{W} \) are upper labels for \( \mathbf{t} \) and lower labels for \( \mathbf{u} \).
\end{proposition}

An immediate corollary is the characterisation of atoms and co-atoms in the lattice of torsion/torsionfree classes:

\begin{corollary}[{\cite[Proposition 3.5]{AsaiPfeifer}}]
The torsion classes covering the zero torsion class are in bijection with the set of isoclasses of simple objects of \( \mathcal{A} \). 

In particular all such classes have the form \(\filt(S) \) for some simple object $ S $.
\end{corollary}

We finally recall the standard finiteness criterion from \cite{DIJ}, extended to abelian length categories, for instance by Enomoto:

\begin{theorem}[{\cite[Theorem 5.5]{EnomotoMonobrick}}]\label{thm:brick-finite-torsion-finite}
An abelian length category is brick-finite (that is, it contains a finite number of isomorphism classes of bricks) if and only if the set of torsion pairs is finite.
\end{theorem}

\section{Exchange of labels across a wide interval}

 Throughout the section we fix two torsion pairs
\[
(\mathbf t,\mathbf f),\qquad (\mathbf u,\mathbf v),
\]
with
\[
\mathbf t\subsetneq \mathbf u,
\]
being a wide interval 
and we denote by \(\mathcal{W} = \mathbf{u} \cap \mathbf{f} \) the corresponding wide subcategory. Thus as we have recalled the simple objects of \(\mathcal{W}\) are torsionfree, almost torsion for \((\mathbf t,\mathbf f)\) and torsion, almost torsionfree for \((\mathbf u,\mathbf v)\).

It is a natural question to ask whether also the other labels of the lower torsion class are in any relation with labels of the upper torsion class. 

The results contained in this section are probably well-known to experts, at least in the context of finite-dimensional algebras. The collection of upper and lower labels of a given torsion class forms what is called a \emph{semibrick pair} in \cite{BH}, where mutation of such collections is also described in the derived category, in the context of  \( \tau-\)tilting finite algebras.

More recently, the same kind of explicit mutation procedure which we will use has been described in \cite[Theorem 2.5]{MutSimple} with an in depth discussion of mutability criteria in terms of existence of certain minimal left approximations.

Without referring to the derived category, we show that the necessary approximations exist  if we assume that  the torsion pair \( \tpair{t}{f} \) is widely generated and \( \tpair{u}{v} \) is widely cogenerated.

We start in the next subsection with a way to relate upper labels of \( \mathbf{u} \) with labels of \( \mathbf{t} \).

\subsection{A correspondence for upper labels}

Recall the following definition:

\begin{definition}
Let \(\mathcal C\) be a full subcategory of \(\A\). A morphism \(e:X\to C\), with \(C\in\mathcal C\), is a \(\mathcal C\)-envelope if every morphism \(X\to C'\), with \(C'\in\mathcal C\), factors through \(e\), and every endomorphism \(h:C\to C\) with \(he=e\) is an automorphism.

Dually, one defines \( \mathcal C\)-covers.
\end{definition}

We will use several times the following elementary observation:

\begin{lemma}
\label{lem:approxim-lem}
Let $ \B $ be a semibrick and let $ \mathcal{C} = \filt(\B) $. Consider the short exact sequence
\[
0 \to X \to C \to Y \to 0
\]
with $ C \in \mathcal{C} $. Then if $ \Hom(Y, \B) = \Ext^1(Y, \B) = 0 $, the map $ X \to C $ is a $ \mathcal{C}-$envelope.

Moreover in a short exact sequence of the form:
\[
0 \to X \to Y \to C \to 0
\]
with $ C \in \mathcal{C} $, if $ \Hom(X, \B) = 0 $ the map $ Y \to C $ is a $ \mathcal{C}-$envelope.
\end{lemma}

\begin{proof}
As every object $ N \in \mathcal{C} $ admits a finite filtration with factors isomorphic to bricks in $ \B $, we deduce that for all such objects  $ \Hom(Y, N) = \Ext^1(Y, N) = 0 $. Applying the functor $ \Hom(-, N) $ to the short exact sequence in the statement we get an isomorphism $ \Hom(C, N) \simeq \Hom(X, N) $.  

This shows that every map $ X \to N $ factors through the given map $ f :  X \to C $. Moreover, if for some map $ e : C \to C $ we have $ e \circ f = f $, that is $ (1_C - e) \circ f = 0 $, but then as the map $ \Hom(C, C) \to \Hom(X, C) $ is an isomorphism, and $ 0 \circ f = 0 $ we must have $ e = 1_C $, in particular it is an automorphism.

The same argument applies for the second sequence, after computing the long exact sequence given by the functor $ \Hom(-, N) $.
\end{proof}

The exchange process is based on the following Proposition whose proof suggests a label mutation algorithm:

\begin{proposition}
\label{prop:exchange-algo}
Assume that \(\mathbf t\) is widely generated. Let \(V\)  be torsionfree, almost torsion for \(\tpair{u}{v}\), and suppose that there is an exact sequence
\[
\begin{tikzcd}
0\ar[r] & K \ar[r] & V\ar[r] & M \ar[r] & 0
\end{tikzcd}
\]
with \(M\in\mathcal{W}\) and with \(K\) a \(\mathbf v\)-simple. Then there exists a \( \mathcal{W}-\)envelope \( e : V \to W_V \) and one of the following occurs
\begin{enumerate}
\item The morphism \(e \) is an epimorphism and \(\ker e \) is a torsionfree, almost torsion object for \( \tpair{t}{f} \).
\item The morphism \(e \) is a monomorphism and \( \coker e \) is a torsion, almost torsionfree object for \( \tpair{t}{f} \).
\end{enumerate}
\end{proposition}

\begin{proof}
We will do an induction on the length of the kernel \( K \) and divide the proof in three steps. The first two steps are independent of the length of \( K \). Let \( \B \) be the semibrick corresponding to \( \mathcal{W} \).

\noindent\textbf{Step 1}: If \( \Hom(K, \B) = 0 \) then the map \( V \to M \) is already a \( \mathcal{W}-\) envelope by Lemma \ref{lem:approxim-lem}. It remains to show that \( K \) is torsionfree, almost torsion with respect to \( \tpair{t}{f} \). 

First, we show that it is \( \mathbf{f}-\)simple. It is non-zero and contained in $ \mathbf{v} \subset \mathbf{f} $. So consider a proper quotient \( K \to Q \). As \( K \) is \( \mathbf{v}-\)simple, \( Q \in \mathbf{u} \). However, \( K \in {}^{\perp_0} \mathcal{W} \), which is a torsion class. So \( Q \in \mathbf{u} \cap {}^{\perp_0} \mathcal{W} \) which is  \( \mathbf{t} \) by Proposition \ref{prop:cover-intersection}.

Now we need to show maximality among \( \mathbf{f}-\)simples. So let $ K \to K' $ be a non-zero map to a relative simple, this map is necessarily a monomorphism. Assume it is proper. Form the pushout:
\[
\begin{tikzcd}
0 \ar[r] & K \ar[r] \ar[d, hook] & V \ar[r] \ar[d, hook] & M \ar[r] \ar[d, equals] & 0 \\
0 \ar[r] & K' \ar[d, two heads] \ar[r] & P \ar[d, two heads] \ar[r] & M \ar[r] & 0 \\
& C \ar[r, equals] & C
\end{tikzcd}
\]
Then $ P \in \mathbf{f} $, being an extension of $ M \in \mathcal{W} \subset \mathbf{f} $ and $ K' $. Moreover \( C \in \mathbf{t} \) being a proper quotient of \( K' \).

 Notice that the short exact sequence \( 0 \to V \to P \to C \to 0 \) is non-split, otherwise \( C = 0 \) as \( P \in \mathbf{f} \), thus being \( V \) torsionfree, almost torsion \( P \in \mathbf{u} \).

In conclusion \( P \in \mathbf{f} \cap \mathbf{u} = \mathcal{W} \). But then the map \( K \to  V \to P \) is the zero map, as \( \Hom(K, \mathcal{W}) = 0 \), which is a contradiction as it is the composite of two non-zero monomorphisms. This shows that the monomorphism \( K \to K' \) can't be proper, so \( K \) is maximal among \( \mathbf{f}-\)simples.

\noindent \textbf{Observation}: If \( f : K \to B \) is non-zero, for some \( B \in \mathcal{B} \) then it is either a (proper) monomorphism or a (proper) epimorphism. 

It clearly can't be an isomorphism, as \( K \in \mathbf{v} \) and \( B \in \mathbf{u} \). So assume it has a non-zero kernel, then the image of \( f \) is a proper quotient of \( K \), thus it lies in \( \mathbf{u} \) and is non-zero. But every proper subobject of \( B \) lies in \( \mathbf{v} \) as \( B \) is \( \mathbf{u}-\)simple. Thus \( f \) must be a proper epimorphism. 

\noindent \textbf{Step 2}: Assume we have a proper monomorphism \( K \to B \) for some \( B \in \B \). Then consider once again the pushout diagram
\[
\begin{tikzcd}
0 \ar[r] & K \ar[r] \ar[d, hook] & V \ar[r] \ar[d, hook] & M \ar[r] \ar[d, equals] & 0 \\
0 \ar[r] & B \ar[d, two heads] \ar[r] & P \ar[d, two heads] \ar[r] & M \ar[r] & 0 \\
& Q \ar[r, equals] & Q
\end{tikzcd}
\]
Notice that \( P \in \mathcal{W} \). I claim that \( Q \) is \( \mathbf{t}-\)simple. It is non-zero and contained in \( \mathbf{t} \) being a proper quotient of \( B \). Now it is enough to show that it has no proper non-zero subobject in \( \mathbf{t} \). Assume \( T \to Q \) is a non-zero subobject. Then, pulling back along \( P \to Q \) we obtain:
\[
\begin{tikzcd}
0 \ar[r] & V \ar[r] \ar[d, equals] & R \ar[r] \ar[d, hook] & T \ar[r] \ar[d, hook] & 0 \\
0 \ar[r] & V  \ar[r] & P \ar[d, two heads] \ar[r] & Q \ar[r] \ar[d, two heads] & 0 \\
& & C \ar[r, equals] & C
\end{tikzcd}
\]
Notice that \( R \in \mathbf{f} \) as it is a subobject of  \(  P \in \mathcal{W} \). Moreover, as the sequence \( 0 \to V \to R \to T \to 0 \) is non-split, \( R \in \mathbf{u} \) as \( V \) is torsionfree, almost torsion for \( \tpair{u}{v} \) and \( T \in \mathbf{t} \subset \mathbf{u} \). Thus \( R \in \mathcal{W} \), but as this is a wide subcategory, so is the cokernel \( C \). But \( C \) should be in the torsion class \( \mathbf{t} \) as it is (isomorphic to) a quotient of \( Q \). Thus \( C = 0 \) and \( T \) is not a proper subobject. 

Now \( Q \) may not be torsion, almost torsionfree, however, as we assumed \( \mathbf{t} \) to be widely generated, there exists a torsion, almost torsionfree object \( \widetilde{Q} \) that has \( Q \) as a quotient. 

Construct the pullback
\[
\begin{tikzcd}
& & L \ar[r, equals] \ar[d, hook] & L \ar[d, hook] \\
0 \ar[r] & V  \ar[d, equals] \ar[r] & M' \ar[d, two heads] \ar[r] & \widetilde{Q} \ar[r] \ar[d, two heads] & 0 \\
0 \ar[r] & V  \ar[r] & P \ar[r] & Q \ar[r] & 0
\end{tikzcd}
\]
\( L \) is a possibly zero, proper subobject of a torsion, almost torsionfree, thus it is contained in \( \mathbf{f} \). This means that \( M' \) is in \( \mathbf{f} \) as an extension of torsionfree objects, moreover \( M' \in \mathbf{u} \) as part of a non-split short exact sequence with \( \widetilde{Q} \in \mathbf{u} \) and \( V \) torsionfree, almost  torsion as above. This shows \( M' \in \mathcal{W} \) and we obtain that \( V \to M' \) is the required envelope using Lemma \ref{lem:approxim-lem} after noticing that the required orthogonality conditions hold for the torsion, almost torsionfree object \( \widetilde{Q} \) and the simple objects of \( \mathcal{W} \) which are torsionfree, almost torsion for the same torsion pair. 

\noindent \textbf{Step 3}: Here we use the length of \( K \). If it is minimal (that is 1, as \( K \ne 0 \) ) then \( K \) is simple and we can't have a proper epimorphism \( K \to B \) for any \( B \in \B \). Thus this completes the base case of induction. So assume the algorithm works for lengths up to \( n \) and assume we have a \( K \) with a proper epimorphism onto some \( B \). We will construct a new short exact sequence \( 0 \to K' \to V \to M' \to 0 \) with \( K' \) being \( \mathbf{v}-\)simple of strictly lower length and \(M' \in \mathcal{W} \), so that we can conclude by induction hypothesis.  

This is once again obtained by taking a pushout of the epimorphism \( K \to B \):
\[
\begin{tikzcd}
 & K' \ar[r, equals] \ar[d, hook] & K' \ar[d, hook] \\
0 \ar[r] & K \ar[d, two heads] \ar[r] & V \ar[d, two heads] \ar[r] & M \ar[r] \ar[d, equals] & 0 \\
0 \ar[r] & B  \ar[r] & M' \ar[r] & M \ar[r] & 0
\end{tikzcd}
\] 
Notice that \( M' \in \mathcal{W} \) as required and \( \ell(K') < \ell(K) \). So it is enough to show that \( K' \) is \( \mathbf{v}-\)simple. But this is immediate: take a proper quotient \( S \) of \( K' \), once again this induces a short exact sequence \( 0 \to S \to S' \to B \to 0 \), where \( S' \in \mathbf{u} \) being a proper quotient of the \( \mathbf{v}-\)simple \( K \). Thus \( S \) is in \( \mathbf{u} \) as \( B \) is torsion, almost torsionfree for \( \tpair{u}{v} \). 
\end{proof}

\begin{corollary}
\label{cor:up-exchange}
Assume that \(\mathbf t\) is widely generated. Let \(V\)  be torsionfree, almost torsion for \(\tpair{u}{v}\)
Then there exists a \( \mathcal{W}-\)envelope \( e : V \to W_V \) and one of the following occurs
\begin{enumerate}
\item The morphism \(e \) is an epimorphism and \(\ker e \) is a torsionfree, almost torsion object for \( \tpair{t}{f} \).
\item The morphism \(e \) is a monomorphism and \( \coker e \) is a torsion, almost torsionfree object for \( \tpair{t}{f} \).
\end{enumerate}
\end{corollary}

\begin{proof}
Apply Proposition \ref{prop:exchange-algo} to the trivial short exact sequence \[ 0 \to V \to V \to 0 \to 0. \] Of course \( V \) is \( \mathbf{v}-\)simple and \( 0 \in \mathcal{W} \).
\end{proof}

\begin{remark}
Notice that in the corollary above, if \( \Hom(V, \mathcal{W}) = 0 \) then the envelope is the zero map and \( V \) itself is a torsionfree, almost torsion object for \( \tpair{t}{f} \).
\end{remark}

\subsection{Exchange for lower labels}

We now show how to associate lower labels of \( \mathbf{u} \) to lower labels of \( \mathbf{t} \).

\begin{proposition}
\label{prop:lower-exchange}
Assume that \(\mathbf t\) is widely generated. Let \(U\)  be torsion, almost torsionfree for \(\tpair{u}{v}\). Then if \( U \not \in \mathcal{W} \), there exists a short exact sequence
\[
0 \to W_T \to T \to U \to 0 
\]
where the map \( W_T \to T \) is a \( \mathcal{W}-\)cover and \( T \) is torsion, almost torsionfree for \( \tpair{t}{f} \). Moreover, if \( T' \) is a torsion, almost torsionfree object for \( \tpair{t}{f} \) not isomorphic to \( T \), we have \( \Hom(T', U) = 0 \).
\end{proposition}

\begin{proof}
First, notice that all the torsion, almost torsionfree objects of \( \tpair{u}{v} \) not contained in \( \mathcal{W} \) are contained in \( {}^{\perp_0}\mathcal{W} \cap \mathbf{u} = \mathbf{t} \) ( see Proposition \ref{prop:cover-intersection}) and are in particular, \( \mathbf{t}-\)simple, as all their proper subobjects lie in \( \mathbf{v} \subset \mathbf{f} \). 

Now, if such an object \( U \) is already a torsion, almost torsionfree object for \( \tpair{t}{f} \), then we are done, as the trivial short exact sequence \( 0 \to 0 \to U \to U \to 0 \) has the claimed properties and Hom-orthogonality to non-isomorphic torsion, almost torsionfree object is immediate by Proposition \ref{prop:maximal-f-simple}. 

If it is not, then as \( \mathbf{t} \) is widely generated, we have some torsion, almost torsionfree object \( T \) that has \( U \) as a proper quotient. Consider the corresponding short exact sequence
\[
0 \to K \to T \to U \to 0
\]
As \( U \) is torsion, almost torsionfree for \( \tpair{u}{v} \), the kernel lies in \( \mathbf{u} \), but as \( T \) is in particular \( \mathbf{t}-\)simple \( K \in \mathbf{f} \). So \( K \in \mathcal{W} \) as required.
As \( \Hom(\mathcal{W}, U) = 0 \), the map \( K \to T \) is a \( \mathcal{W}-\)cover, by the dual of Lemma \ref{lem:approxim-lem}.

Now, let \( T' \) be a torsion, almost torsionfree object, not-isomorphic to \( T \). Then, applying the functor \( \Hom(T', -) \) to the approximation sequence we get the exact sequence
\[
\Hom(T', T) \to \Hom(T', U) \to \Ext^1(T', K)
\]
but \( \Hom( T', T) = 0 \) as these are non-isomorphic lower labels, moreover \( \Ext^1(T', B) = 0 \) for all the simple objects of the wide subcategory \( \mathcal{W} \) by Lemma \ref {lem:ext-orthogonality} so that also \( \Ext^1(T', K) = 0 \) as \( K \in \mathcal{W} \) is filtered by the simple \(\mathcal{W}-\)objects.
Thus \( \Hom(T', U) = 0 \).
\end{proof}

\subsection{The bijection of labels}

We write down explicitly the dual results before giving the full exchange theorem:

\begin{corollary}
\label{cor:up-exchange-dual}
Assume that \(\mathbf v\) is widely cogenerated. Let \(T\)  be torsion, almost torsionfree for \(\tpair{t}{f}\)
Then there exists a \( \mathcal{W}-\)cover \( c : W_T \to T \) and one of the following occurs
\begin{enumerate}
\item The morphism \(c \) is an epimorphism and \(\ker c \) is a torsionfree, almost torsion object for \( \tpair{u}{v} \).
\item The morphism \(c \) is a monomorphism and \( \coker c \) is a torsion, almost torsionfree object for \( \tpair{u}{v} \).
\end{enumerate}
\end{corollary}

\begin{proposition}
\label{prop:upper-exchange-dual}
Assume that \(\mathbf v\) is widely cogenerated. Let \(F\)  be torsionfree, almost torsion for \(\tpair{t}{f}\). Then if \( F \not \in \mathcal{W} \), there exists a short exact sequence
\[
0 \to F \to V \to W_V \to 0 
\]
where the map \( V \to W_V \) is a \( \mathcal{W}-\)envelope and \( V \) is torsionfree, almost torsion for \( \tpair{u}{v} \). Moreover, if \( V' \) is a torsionfree, almost torsion object for \( \tpair{u}{v} \) not isomorphic to \( V \), we have \( \Hom(F, V') = 0 \).
\end{proposition}

\begin{theorem}
\label{thm:full-mutation-labels}
Let \( \tpair{t}{f} \) and \( \tpair{u}{v} \) be torsion pairs with \( \mathbf{t} \subsetneq \mathbf{u} \) a wide interval and such that the lower pair is widely generated, while the upper pair is widely cogenerated.
Then there is a bijection
\[
\Phi : \Up(\mathbf u)\coprod \Low(\mathbf u) \to \Up(\mathbf t)\coprod \Low(\mathbf t).
\]
defined as follows:
\begin{enumerate}
\item for an upper label $ [V] $ of $ \mathbf{u} $, take a $ \mathcal{W}-$envelope $ e_V : V \to W_V $. This is either a monomorphism or an epimorphism, if mono, $ \Phi([V]) := [\coker e_V] $ if epi $ \Phi([V]) := [\ker e_V] $.
\item for a lower label $ [U] $, if $ U \in \mathcal{W} $, then $ \Phi([U]) := [U] $ otherwise we fix $ \Phi([U]) := [T] $ where $ T $ is the unique torsion, almost torsionfree object with \( \Hom(T, U) \ne 0 \).
\end{enumerate}
\end{theorem}

\begin{proof}
First we notice that the map is well defined: the associated objects exist and are upper and lower labels of \( \tpair{t}{f} \) by Corollary \ref{cor:up-exchange} and  Proposition \ref{prop:lower-exchange}, moreover isomorphic objects have isomorphic envelopes and via the factorisation property, isomorphic kernels/cokernels.

To show it is a bijection, we consider the following map and show it is its inverse.
Let \( \Psi :  \Up(\mathbf t)\coprod \Low(\mathbf t) \to  \Up(\mathbf u)\coprod \Low(\mathbf u) \) be defined as follows:
\begin{enumerate}
\item  for a lower label $ [T] $, take  a \(\mathcal{W}-\)cover \( c_T : W_T \to T \).  This is either a monomorphism or an epimorphism, if mono, $ \Psi([T]) := [\coker c_T] $ if epi $ \Psi([T]) := [\ker c_T] $.
\item for an upper label \( [F] \), if \( F \in \mathcal{W} \) then \( \Psi([F]) = [F] \), otherwise we fix \( \Psi([F]) = [V] \) where \( V \) is the unique (up to isomorphism) torsionfree, almost torsion object for \( \tpair{u}{v} \) admitting a non-zero map from \( F \).
\end{enumerate}

Once again, this map is well defined by Corollary \ref{cor:up-exchange-dual} and  Proposition \ref{prop:upper-exchange-dual}. 

Let's compute \( \Psi(\Phi([V])) \) for an upper label of \( \tpair{u}{v} \). If the envelope is a monomorphism, then we get a lower label \(  [\coker e_V] \) of \( \tpair{t}{f} \). But observe that the short exact sequence
\[
0 \to V \to W_V \to \coker e_V \to 0
\]
also gives a \( \mathcal{W}-\)cover of \( \coker e_V \) by the dual of Lemma \ref{lem:approxim-lem}. So \( [\coker e_V] \) is a lower label with epimorphic cover and \( \Psi( [\coker e_V] ) = [ V] \).

On the other hand, if the envelope is an epimorphism, we have the short exact sequence
\[
0 \to \ker e_V \to V \to W_V \to 0
\]
and \( [ \ker e_V ] \) is an upper label for \( \tpair{t}{f} \). Notice that \( \ker e_V \not \in \W \) as \( \Hom(\mathcal{W}, V) = 0 \), moreover \( V \) is the only torsionfree, almost torsion object up to isomorphism with \( \Hom( \ker e_V, V) \ne 0 \). This is clear, applying to the short exact sequence the functor \( \Hom(-, V') \) and using the usual Hom and Ext-orthogonality properties. Thus \( \Psi([\ker e_V ]) = [V] \). 

Let's now compute  \( \Psi(\Phi([U])) \) for a lower label of \( \tpair{u}{v} \). If \( [U] \in \mathcal{W} \) then we are done. So assume \( U \not \in \mathcal{W} \) and let \( [T] \) be the associated lower label of \( \tpair{t}{f} \). In every short exact sequence
\[
0 \to W_T \to T \to U \to 0
\] 
we now know that the map \( W_T \to T \) is a \( \mathcal{W}-\)cover as in the proof of Proposition \ref{prop:lower-exchange}. But then \( \Psi([T]) = [U] \). 

This shows that \( \Psi \circ \Phi = \mathrm{Id} \). One checks, by dual computations that also \( \Phi \circ \Psi = \mathrm{Id} \).
\end{proof}

\section{Connected lattices of torsion classes}

As an intermediate step towards our characterisation, we consider the following property: in a brick-finite category for every torsion class \( \mathbf{t} \) we can find a finite collection of consecutive cover relations connecting it with the zero torsion class: we call such a chain of torsion classes \( 0 \lessdot \mathbf{t}_1 \dots \lessdot  \mathbf{t}_n = \mathbf{t} \) saturated.

\begin{definition}
Let \( \A \) be an abelian length category. We say that \( \A \) is \emph{upper torsion connected} if for every torsion class \( \mathbf{t} \) there exists a finite saturated chain connecting \( 0 \) and \( \mathbf{t} \).
\end{definition}

\begin{lemma}
\label{lem:semibrickFin}
If \( \A \) is upper torsion connected then every torsion class is compact and every semibrick is finite.
\end{lemma}

\begin{proof}
First notice that if every torsion class is compact, then every semibrick must be finite, in fact, for any given infinite semibrick \( \mathcal{B} \) the torsion class \( \tGen{\mathcal{B}} \) is not compact. 

So we just need to prove the first part of the statement. Let \( \mathbf{t} \) be a torsion class. We proceed by induction on the minimal length of a saturated chain connecting with zero. Clearly, in the case \( n = 0 \) we have the unique possibility \( \mathbf{t} = 0 \) which is compact. 

If we have a chain of length \( n + 1 \), say \( 0 \lessdot \mathbf{t}_1 \lessdot \dots \lessdot \mathbf{t}_n \lessdot \mathbf{t}_{n + 1} = \mathbf{t} \) then by induction, we have that \( \mathbf{t}_n \) is compact. So it has the form \( \tGen{M_n} \) for some object. But then, if \( B_n \) is the brick label of the cover \( \mathbf{t}_n \lessdot \mathbf{t} \) we have \( \mathbf{t} = \tGen{M_n \oplus B_n} \), proving compactness. 
\end{proof}

This observation is enough to show that the only torsion connected categories are the brick-finite ones:

\begin{proposition}
\label{prop:torsConnIsBrickFin}
The following statements are equivalent for an abelian length category \( \A \):
\begin{enumerate}
\item \( \A \) is brick-finite.
\item \( \A \) is upper torsion connected.
\end{enumerate}
\end{proposition}

\begin{proof}
We prove the non-trivial implication. Assume \( \A \) is upper torsion connected. Assume \( \A \) is brick-infinite, so that there are infinitely many torsion classes. Consider the tree whose vertices are the saturated paths starting at zero. We add an edge \( p_1 \to p_2 \) if \( p_2 \) is obtained from \( p_1 \) adding a cover relation on top. 

Notice that this tree is infinite, as every torsion class admits a finite chain connecting it with zero and we assumed an infinite number of torsion classes. Moreover, this tree is finitely branching: in fact every chain ending at some torsion class \( \mathbf{t} \) can only be connected to chains indexed by the torsionfree, almost torsion objects for the torsion pair corresponding to  \( \mathbf{t} \). However, the torsionfree, almost torsion objects for a fixed torsion pair form a semibrick, and all semibricks in \( \A \) are finite by Lemma \ref{lem:semibrickFin}.

Thus we may apply Koenig's Lemma to obtain an infinite branch in this tree. This yields an infinite saturated chain of torsion classes \( 0 \lessdot \mathbf{t}_1 \lessdot \dots \mathbf{t}_n \lessdot \dots \). But then \( \bigvee_n \mathbf{t}_n \) is a non-compact torsion class in \( \A \): if compact it would be equal to some of the \( \mathbf{t}_i \). This contradicts the fact that every torsion class in \( \A \) is compact, again by Lemma \ref{lem:semibrickFin}.
\end{proof}

\section{Local and global finiteness}

We now use the exchange relation, specialised to coverings, to prove local finiteness of the Hasse quiver at all torsion classes lying at finite distance from the zero class. Then a Zorn's Lemma argument shows that every torsion class is at finite distance from zero.

\subsection{Local finiteness along finite paths}

\begin{lemma}\label{lem:local-finiteness}
Let \(\A\) be a WD and WCD abelian length category of  finite rank \( N \). If there exists a saturated chain
\[
0=\mathbf t_0\lessdot\mathbf t_1\lessdot\cdots\lessdot\mathbf t_n=\mathbf t,
\]
then \( \mathbf{t} \) has exactly \( N \) neighbours in the Hasse quiver.
\end{lemma}

\begin{proof}
We argue by induction on \(n\).

If \(n=0\), then \(\mathbf t=0\). The upper labels of the zero torsion class are precisely the isomorphism classes of simple objects of \(\A\). There are no lower labels.

Assume the statement holds for paths of length \( n \), then if we have a torsion class connected to zero via a path of length \( n + 1 \), it is enough to apply Theorem \ref{thm:full-mutation-labels} to the cover \( \mathbf{t}_n \lessdot \mathbf{t}_{n + 1} \). This is a wide interval and all the involved torsion pairs are widely generated/ cogenerated by assumption.
\end{proof}

We recover the following, which seems to be also an easy consequence of \cite[Theorem 5.3]{BPPW}.

\begin{corollary}
Let $ \mathcal{A} $ be a brick-finite abelian length category.
Then the Hasse diagram of $ \tors(\A) $ is regular.
\end{corollary}

\begin{proof}
Every brick-finite abelian length category is WD and WCD and has finite rank. Moreover, every torsion class can be reached along a finite path from the zero class.
Thus Lemma \ref{lem:local-finiteness} applies to every torsion class.
\end{proof}

\subsection{Finite distance from zero}

We need one elementary lattice observation.

\begin{lemma}\label{lem:cocompact-chain}
Let \(L\) be a complete lattice and let \((x_i)_{i\in I}\) be a strictly descending chain with no minimal element. If
\[
x=\bigwedge_{i\in I}x_i,
\]
then \(x\) is not co-compact.
\end{lemma}

\begin{proof}
Co-compact means compact in the opposite lattice. If \(x\) were co-compact, then from
\[
\bigwedge_{i\in I}x_i\leq x
\]
we would get a finite subset \(J\subseteq I\) such that
\[
\bigwedge_{j\in J}x_j\leq x.
\]
The other inequality is automatic, hence equality holds. Since the family is a chain, the finite meet \(\bigwedge_{j\in J}x_j\) is one of the \(x_j\). Thus \(x=x_j\) for some \(j\), contradicting the assumption that the chain has no minimal element.
\end{proof}

\begin{lemma}\label{lem:finite-distance}
Let \(\A\) be a WD and WCD abelian length category of finite rank. Then every torsion class in \(\A\) can be reached from the zero torsion class by a finite saturated chain.
\end{lemma}

\begin{proof}
Let \(\mathcal B\) be the set of bad torsion classes, namely those torsion classes which cannot be reached from \(0\) by a finite saturated chain. We show that \(\mathcal B\) is empty.

Notice that if not empty, the set \( \mathcal B \) is co-inductive.
Assume first that \((\mathbf b_i)_{i\in I}\) is a descending chain in \(\mathcal B\). If the chain has a minimal element, then that element is a lower bound in \(\mathcal B\). Otherwise set
\[
\mathbf b=\bigwedge_{i\in I}\mathbf b_i.
\]
By Lemma \ref{lem:cocompact-chain}, \(\mathbf b\) is not co-compact. Since \(\A\) is WCD, Corollary \ref{cor:basic-wide-consequences} implies that \(\mathbf b\) has infinitely many upper covers. But any torsion class at finite distance from zero has only finitely many upper covers by Lemma \ref{lem:local-finiteness}. Therefore \(\mathbf b\) is itself bad. Thus every descending chain in \(\mathcal B\) has a lower bound in \(\mathcal B\).

By Zorn's lemma, \(\mathcal B\) has a minimal element \(\mathbf b_{\min}\). It is not the zero class. Since \(\A\) is WD, Corollary \ref{cor:basic-wide-consequences} gives a lower cover
\[
\mathbf b'\lessdot\mathbf b_{\min}.
\]
If \(\mathbf b'\) were at finite distance from zero, then so would \(\mathbf b_{\min}\), by adding the last cover. Hence \(\mathbf b'\in\mathcal B\), contradicting the minimality of \(\mathbf b_{\min}\). This contradiction proves that \(\mathcal B=\varnothing\).
\end{proof}

\begin{theorem}\label{thm:main}
Let \(\A\) be a finite-rank abelian length category. The following statements are equivalent.
\begin{enumerate}[label=\textup{(\arabic*)},leftmargin=2.5em]
\item \(\A\) is brick-finite.
\item \(\A\) is widely determined and widely co-determined.
\end{enumerate}
\end{theorem}

\begin{proof}
Assume first that \(\A\) is brick-finite. Then there are only finitely many torsion pairs by Theorem \ref{thm:brick-finite-torsion-finite}. In particular, all torsion and torsionfree classes are compact. By Proposition \ref{prop:finitely-cogenerated}, every torsion class is widely generated and every torsionfree class is widely cogenerated. Thus \(\A\) is WD and WCD.

Conversely, assume that \(\A\) is WD and WCD. By Lemma \ref{lem:finite-distance}, the category \( \A \) is upper torsion connected. By Proposition \ref{prop:torsConnIsBrickFin}, \( \A \) is brick-finite.
\end{proof}

As a consequence, we obtain the announced Brick Brauer-Thrall statement.

\begin{theorem}\label{thm:bounded-length}
Let \(\A\) be a finite-rank abelian length category. Then \(\A\) is brick-finite if and only if there exists \(N\in\mathbb N\) such that every brick \(B\in\A\) satisfies \(\elll(B)<N\).
\end{theorem}

\begin{proof}
If \(\A\) is brick-finite, then the set of lengths of bricks is finite, hence bounded.

Conversely, assume that there exists \(N\in\mathbb N\) such that \(\elll(B)<N\) for every brick \(B\). We show that \(\A\) is WD and WCD.

Let \((\mathbf t,\mathbf f)\) be a torsion pair, and let \(S\) be an \(\mathbf f\)-simple object. If \(S\) is not torsionfree, almost torsion, then by Proposition \ref{prop:maximal-f-simple} there exists an \(\mathbf f\)-simple object \(S_1\) and a non-zero non-isomorphism \(S\to S_1\). Since \(S\) is \(\mathbf f\)-simple, this map is a monomorphism, hence
\[
\elll(S_1)>\elll(S).
\]
Repeating the argument, and using the uniform bound on lengths of bricks, the process must stop. Thus every \(\mathbf f\)-simple object embeds into a torsionfree, almost torsion object. By Proposition \ref{prop:wide-cogenerated-criterion}, \(\mathbf f\) is widely cogenerated.

The dual argument shows that every \(\mathbf t\)-simple object is a quotient of a torsion, almost torsionfree object. Hence every torsion class is widely generated. Therefore \(\A\) is WD and WCD, and Theorem \ref{thm:main} implies that \(\A\) is brick-finite.
\end{proof}

\begin{remark}
Notice that in the proof of Theorem \ref{thm:main} and of the necessary Lemmas it would be enough to assume a slightly weaker condition: a finite rank abelian length category is brick finite if and only if it is WD and every torsionfree class that can be reached from the zero class along a finite path in the Hasse quiver is widely cogenerated. 
Notice that the second condition always holds in the module category of an artin algebra, where the one-sided WD or WCD condition is in fact enough to ensure brick-finiteness (see \cite[Proposition 5.4]{WideCoreflective}).
\end{remark}

\section{Further directions}

It is clear that the finite-rank hypothesis is necessary for Theorem \ref{thm:main} (in fact as simple objects are bricks, without finite rank it is impossible to have brick-finiteness). Nevertheless, it is still interesting to consider abelian length categories with infinitely many simple objects which are both widely determined and widely co-determined. Such categories often have well-behaved lattices of torsion pairs, even when they are brick-infinite.

\begin{example}
Let \(k\) be an algebraically closed field and consider the category of finite-length \(k[x]\)-modules. Every simple module is one-dimensional and has the form
\[
S_\lambda=k[x]/(x-\lambda),\qquad \lambda\in k.
\]
The simple modules are pairwise non-isomorphic, so there are infinitely many of them. Every finite-length indecomposable is of the form
\[
k[x]/(x-\lambda)^n,
\]
and for \(n>1\) it admits a non-trivial nilpotent endomorphism. Thus the bricks are precisely the simple modules.

Consequently, the category is brick-infinite but every brick has length one. The torsion classes, torsionfree classes and wide subcategories all identify with subsets of \(k\). In particular, the category is both widely determined and widely co-determined, although it is not finite rank.
\end{example}

Also notice that in infinite rank WD and WCD are easily seen to be independent conditions.

\begin{example}
Let $Q$ be the linearly oriented $ \mathbb{A}_\infty $ quiver
\[
1 \to 2 \to 3 \to \cdots
\]
and let $\mathcal{A} = \mathrm{rep}_k(Q)$ be the abelian length category of finite-dimensional $k$-representations.
It is well known that the finite-dimensional indecomposables of $\mathcal{A}$ are the \emph{interval modules}
\[
k_I \qquad I =  [i, j] \text{ where }1 \le i \le j,
\]
with \( k_i := k \) and all structure maps the identity; equivalently, every object of $\mathcal{A}$ decomposes as a finite direct sum of such interval modules \cite[Theorem 1.1]{CrawleyBoevey}.

Moreover, each $k_I$ is uniserial, and its submodules are exactly the terminal subintervals
\[
k_{[t,j]} \subseteq k_{[i,j]} \qquad i \le t \le j.
\]
In particular, if we have a monomorphism \( k_I \to k_J \) then the intervals \( I \) and \( J\) have the same endpoint.

Now let $\mathbf{f}\subseteq \mathcal{A}$ be a torsionfree class. To show that $ \mathbf{f} $ is widely cogenerated it is enough to verify that every $ \mathbf{f}-$simple is contained in some torsionfree, almost torsion object ( see Proposition \ref{prop:wide-cogenerated-criterion} ).

Take a $\mathbf{f}$-simple $B\in\mathbf{f}$.
Then $B\simeq k_{[i,j]}$. If $B\hookrightarrow S$ with $S$ a $\mathbf{f}$-simple, then  we have $S\simeq k_{[k,j]}$ with $1\le k\le i$.
So there are only finitely many possibilities for such $S$, hence $B$ is contained in a maximal $\mathbf{f}$-simple.

On the other hand, consider the full subcategory
\[
\mathbf{t} := \mathrm{add}\{ k_{[1,n]} \mid n\ge 1\} \subseteq \A.
\]
Then $\mathbf{t}$ is a torsion class: closure under quotients is clear, as any indecomposable quotient of an object in \( \mathbf{t} \) must have the form \( k_{[1,m]} \) while closure under extension can be easily obtained for instance noticing that all the given interval representations are injective.  The corresponding torsionfree class is the Serre subcategory $ \mathrm{filt}( S_2, \dots, S_n, \dots) $ generated by all the simple objects \( S_i = k_{[i, i]} \), except for \( S_1 \).

However, the given torsion class is not widely generated: in fact there are no torsion, almost torsionfree objects. First notice that every \( k_{[1, n]} \) is \( \mathbf{t}-\)simple (all its proper subobjects are torsionfree ) then for each $m\ge 1$ the epimorphism
\[
\pi_m\colon k_{[1,m+1]} \twoheadrightarrow k_{[1,m]}
\]
proves that the set of \(\mathbf{t}-\)simples doesn't have any maximal object.
Thus $\mathbf{t}$ is not widely generated.

In conclusion, in $\mathrm{rep}_k(\mathbb{A}_\infty)$ every torsionfree class is widely cogenerated, but there exist torsion classes which are not widely generated.
\end{example}

This motivates the following local finiteness condition.

\begin{definition}
We say that an abelian length category \(\A\) is locally brick-finite if for every simple object \(S\in\A\), there are only finitely many isomorphism classes of bricks having \(S\) as a composition factor.
\end{definition}

\begin{proposition}\label{prop:locally-brick-finite-WD-WCD}
If \(\A\) is locally brick-finite, then \(\A\) is both widely determined and widely co-determined.
\end{proposition}

\begin{proof}
We prove that \(\A\) is WCD; the other statement is dual. Let \((\mathbf t,\mathbf f)\) be a torsion pair and let \(B\) be an \(\mathbf f\)-simple object. Choose a composition factor \(S\) of \(B\). By local brick-finiteness, there are only finitely many bricks having \(S\) as a composition factor. Hence there are only finitely many \(\mathbf f\)-simple objects containing \(B\) as a subobject. Therefore one of them is maximal, hence torsionfree, almost torsion by Proposition \ref{prop:maximal-f-simple}. In conclusion, by Proposition \ref{prop:wide-cogenerated-criterion}, \(\mathbf f\) is widely cogenerated.
\end{proof}

\begin{remark}
Notice that for infinite rank categories, also the link with the first Brauer-Thrall condition becomes weaker: if an abelian length category satisfies the boundedness condition for brick-lengths then it is surely widely determined and co-determined, however, there are locally brick-finite categories that do not satisfy the condition. For instance, fixed a field $ k $, take the direct sum of the categories $ k\mathbb{A}_n-\mathrm{mod} $ for the linearly oriented quiver $ \mathbb{A}_n : \bullet \to \bullet \to \dots \to \bullet $. This category is surely locally brick finite, but its bricks have unbounded lengths.

This category is "disconnected" in the appropriate sense, so to exclude these trivial counterexamples, one should probably investigate the relation between these properties only for connected categories.
\end{remark}

\begin{question}
Is an abelian length category widely determined and widely co-determined if and only if it is locally brick-finite?
\end{question}

\end{document}